\documentclass[11pt,a4paper]{scrartcl}

\usepackage[utf8]{inputenc}
\usepackage{amsmath,amssymb,amsfonts,amsthm}
\usepackage{mathrsfs}
\usepackage{mathabx}
\usepackage{makeidx}
\usepackage{verbatim}
\usepackage[toc,header,page]{appendix}
\usepackage[automark,headsepline]{scrpage2}
\allowdisplaybreaks[1]

\author{Julian Pook \\ Universit\'e Libre de Bruxelles}
\title{Homogeneous and Locally Homogeneous Solutions to Symplectic Curvature Flow}
\date{\today}

\clearscrheadfoot
\ihead{\headmark}
\ohead{\pagemark}
\makeatletter
\makeindex

\begin{document}
 \theoremstyle{definition}
 \newtheorem{definition}{Definition}
 \newtheorem*{fact}{Fact}
 \newtheorem*{remark}{Remark}
 \theoremstyle{lemma}
 \newtheorem{theorem}{Theorem}
 \newtheorem{lemma}{Lemma}
 \newtheorem{corollary}{Corollary}
 \newtheorem{proposition}{Proposition}
 \newtheorem*{criterion}{Criterion}
 \newtheorem{claim}{Claim}

 \newcommand{\sA}{\mathcal{A}}
 \newcommand{\sS}{\mathcal{S}}
 \newcommand{\sC}{\mathcal{C}}
 \newcommand{\sD}{\mathcal{D}}
 \newcommand{\sH}{\mathcal{H}}
 \newcommand{\sI}{\mathcal{I}} 
 \newcommand{\sR}{\mathcal{R}}
 \newcommand{\bR}{\mathbb{R}}
 \newcommand{\bC}{\mathbb{C}} 
 \newcommand{\bN}{\mathbb{N}}
 \newcommand{\bZ}{\mathbb{Z}}  
 \newcommand{\bT}{\mathbb{T}}
 \newcommand{\bM}{\mathbb{M}}
 \newcommand{\1}{\mathbf{1}}
 
 \newcommand{\Ric}{\operatorname{Ric}}
 \newcommand{\ad}{\operatorname{ad}}
 \newcommand{\tr}{\operatorname{tr}}
 \newcommand{\supp}{\operatorname{supp}}
 \renewcommand{\Re}{\operatorname{Re}}
 \renewcommand{\Im}{\operatorname{Im}}
 \renewcommand{\ker}{\operatorname{ker}}
 \newcommand{\ima}{\operatorname{ima}} 
 \newcommand{\aut}{\operatorname{aut}}
 \renewcommand{\span}{\operatorname{span}}

 \newcommand{\wt}{\widetilde}
 \newcommand{\ol}{\overline}
 \newcommand{\into}{\hookrightarrow}
 \newcommand{\onto}{\twoheadrightarrow}

\maketitle


\centerline{  {\bf Abstract} }
In \cite{StreetsTian} J. Streets and G. Tian introduce symplectic curvature flow (SCF), a geometric flow on almost K\"ahler manifolds generalising K\"ahler--Ricci flow. The present article gives examples of explicit solutions to SCF of non-K\"ahler structures on several nilmanifolds and on twistor fibrations over hyperbolic space studied by J. Fine and D. Panov in \cite{FinePanov1}. The latter lead to examples of non-K\"ahler static solutions of SCF which can be seen as analogues of K\"ahler--Einstein manifolds in K\"ahler--Ricci flow.

\section{Introduction}

The aim of this article is to present explicit non-K\"ahler solutions to symplectic curvature flow recently introduced by  J. Streets and G. Tian in \cite{StreetsTian}. 
Symplectic curvature flow on an almost K\"ahler manifold $(M^{2n},\omega_0,J_0)$ is given by a system of coupled evolution PDEs for the symplectic structure $\omega$ and the almost complex structure $J$ with initial conditions $\omega(0) = \omega_0, \, J(0) = J_0$. Explicitly, 
\begin{eqnarray*}
  \partial_t \omega &=& -2P \\
  \partial_t J &=& -2 g^{-1}\left[P^{(2,0)+(0,2)} \right] + \mathcal{R} \, .
\end{eqnarray*}
Here, $P$ denotes the Chern--Ricci form given by $2i$ times the curvature of the Chern connection on the almost anti-canonical bundle $\Lambda^{n,0}(TM)$ and $P^{(2,0)+(0,2)}$ is the sum of the $(2,0)$ and $(0,2)$ part of $P$. The musical isomorphism $g^{-1}$ raises the second index, i.e. $g(g^{-1}P^{(2,0)+(0,2)} \xi, \eta) = P^{(2,0)+(0,2)}(\xi,\eta)$. Finally, $\mathcal{R} := [Rc,J]$ is the $J$-anti-linear part of $Rc$, where $Rc$ denotes the Riemann--Ricci curvature tensor $\Ric$ viewed as an endomorphism of the tangent bundle via $g$.    \par
Key properties of this flow proved in \cite{StreetsTian} include parabolicity, short time existence and preservation of the almost K\"ahler property of $\omega$ and $J$. Furthermore, if the initial $J_0$ is integrable, i.e. $(M^{2n},\omega_0,J_0)$ is K\"ahler, then $P$ is the K\"ahler--Ricci form and $\partial_t J=0$, so in this case SCF reduces to K\"ahler--Ricci flow. \par
In Section~\ref{Sec:CpStSol} we show that SCF on certain twistor fibrations over hyperbolic space lead to compact non-K\"ahler static solutions, while Section~\ref{Sec:SCFLeftInvar} presents explicit solutions to SCF of left-invariant structures on select nilmanifolds.  \\ 

\hspace{-14pt} \textbf{Acknowledgements:} The author wishes to thank Joel Fine for the helpful discussions as well as for bringing SCF to his attention, Michel Cahern for pointing out the moment map trick in the proof of Proposition~\ref{Prop:1d} and the FNRS for providing a generous ``Aspirant'' doctoral scholarship.

\section{Compact non-K\"ahler Static Solutions to SCF \label{Sec:CpStSol}}

The SCF equations can be readily solved if $\partial_t \omega = \lambda \omega_0, \, \partial_t J =0, \lambda \in \bR$, in which case the flow acts by rescaling the metric: $\omega(t) = (1+\lambda t) \omega_0, \, J(t) = J_0$. Such solutions are called \emph{static} and in the K\"ahler setting this behaviour is exhibited by K\"ahler--Einstein metrics. We present examples of compact static solutions to SCF in dimensions $n(n+1)$ which cannot be K\"ahler if $n>1$. They are constructed from the twistor fibrations 
\begin{equation*}
 \pi \, : \, Z_{2n} \to H^{2n} \,,
\end{equation*}
where the fibre over each point in $2n$-dimensional hyperbolic space $H^{2n}$ consists of all almost complex structures compatible with the standard hyperbolic metric on $H^{2n}$ inducing a fixed orientation. These spaces are examples of symplectic twistor spaces described by A. Reznikov in \cite{Reznikov}. J. Fine and D. Panov showed in \cite{FinePanov1} that $Z_{2n}$ can be realised as a coadjoint orbit. We follow their approach to define a symplectic structure $\omega$ and a compatible almost complex structure $I$ on $Z_{2n}$ and show that $(Z_{2n},\omega,I)$ is a static solution to SCF. Furthermore, these static solutions descend to compact quotients of $Z_{2n}$ with hyperbolic fundamental group which cannot support any K\"ahler structures if $n>1$. We find that the flow shrinks the metric if $n>2$, expands it if $n=1$ and leaves it invariant in the case $n=2$. 

\subsection{Coadjoint Orbit Description} 

Consider $\mathrm{SO}(2n,1)$, the identity component of the group of isometries of $\bR^{2n+1}$ with Lorentzian metric. Its Lie algebra is given by
\begin{equation*}
 \mathfrak{so}(2n,1) =  \left\{
 \begin{pmatrix}
  0 & \mathbf{u}^t \\ \mathbf{u} & A 
 \end{pmatrix}  \biggl  |  \,  \mathbf{u} \in \bR^{2n}, \, A \in \mathfrak{so}(2n)   \right\} \,.
\end{equation*}
In this description, $\mathrm{SO}(2n)$ can be seen as a subgroup of $\mathrm{SO}(2n,1)$ defined as the stabiliser of $(1,0) \in \bR \times \bR^{2n} = \bR^{2n+1}$.
A choice of almost complex structure $J_0 \in \mathfrak{so}(2n)$ on $\bR^{2n}$ defines an element 
\begin{equation*}
 \xi_0 := \begin{pmatrix}
  0 & 0 \\ 0 & J_0 
 \end{pmatrix} \in \mathfrak{{}so}(2n,1)
\end{equation*}
 and singles out a copy of $\mathrm{U}(n)$ inside $\mathrm{SO}(2n) \subset \mathrm{SO}(2n,1)$ as the stabiliser of $\xi_0$ under the adjoint action (the matrices in $\mathrm{SO}(2n,1)$ commuting with $\xi_0$ are those with $\mathbf{u} = 0$ and $AJ_0 = J_0A$). Set
\begin{equation*}
 \mathcal{O}(\xi_0) \cong \mathrm{SO}(2n,1)/\mathrm{U}(n)
\end{equation*} 
to be the adjoint orbit of $\xi_0$. The Killing form on $\mathfrak{so}(2n,1)$ is non-degenerate and defines an isomorphism $\mathfrak{so}(2n,1) \cong \mathfrak{so}(2n,1)^*$ intertwining the adjoint and coadoint action of $\mathrm{SO}(2n,1)$, so $\mathcal{O}(\xi_0)$ can be seen as a coadjoint orbit. Standard theory then endows  $\mathcal{O}(\xi_0)$ with a $\mathrm{SO}(2n,1)$-invariant symplectic structure $\omega$. \par 
An explicit description of the tangent space of $\mathcal{O}(\xi_0)$ can be given with the help of the following lemma.
\begin{lemma} \label{Lem:Decomp}
As a $\mathrm{U}(n)$ representation space, $\mathfrak{so}(2n,1)$ admits the equivariant  decomposition
\begin{equation*} 
 \mathfrak{so}(2n,1) \cong \mathfrak{u}(n) \oplus \Lambda^2 (\bC^n)^* \oplus \bC^n \, ,
\end{equation*}
where $\bC^n = (\bR^{2n},J_0)$.
\end{lemma}
\begin{proof} 
We only sketch the proof here and refer to \cite{FinePanov1} details. 
Let $U \in \mathrm{U}(n) \subset \mathrm{SO}(2n) \subset \mathrm{SO}(2n,1)$. The adjoint action of $U$ on $\mathfrak{so}(2n,1)$ is given by
\begin{equation*}
 \operatorname{Ad}_U  \begin{pmatrix}
  0 & \mathbf{u}^t \\ \mathbf{u} & A 
 \end{pmatrix}  =
 \begin{pmatrix}
  0 & (U\mathbf{u})^t \\ U \mathbf{u} & UAU^{-1} 
 \end{pmatrix} 
\end{equation*}
from which the equivariant splitting $\mathfrak{so}(2n,1) \cong \mathfrak{so}(2n) \oplus \bC^n$ can be inferred. Those elements in $\mathfrak{so}(2n)$ commuting with $J_0$ constitute $\mathfrak{u}(n)$ as a subset of $\mathfrak{so}(2n)$. The $\mathrm{U}(n)$-invariant complement of $\mathfrak{u}(n)$ in $\mathfrak{so}(2n)$ can be naturally identified with $\Lambda^2(\bC^n)^*$ giving the desired decomposition.   
\end{proof} 

Viewing $\mathcal{O}(\xi_0)$ as $\mathrm{SO}(2n,1)/\mathrm{U}(n)$, the lemma implies that 
\begin{equation*}
 T_{\xi_0} \mathcal{O}(\xi_0) \cong T_{E\cdot \mathrm{U}(n)}( \mathrm{SO}(2n,1) / \mathrm{U}(n) )\cong T_E \mathrm{SO}(2n,1) / T_E \mathrm{U}(n) \cong \Lambda^2 (\bC^n)^* \oplus \bC^n \, ,
\end{equation*}
from which it becomes apparent that the (real) dimension of $\mathcal{O}(\xi_0)$ is $n(n+1)$. 
Since $\mathcal{O}(\xi_0)$ is a homogenous space, the same description is valid for the tangent spaces at other points as well. However, the almost complex structure $J$ determining the identification $\bC^n \cong (\bR^{2n},J)$ will depend on the chosen point. The following consideration makes this clearer. \par
Observe that the different points in the adjoint orbit of $\xi_0$ under $\mathrm{SO}(2n) \subset \mathrm{SO}(2n,1)$ are of the form 
\begin{equation*}
 \xi = \begin{pmatrix} 0 &0 \\ 0& J \end{pmatrix} \, ,
\end{equation*}  
where $J = A J_0 A^{-1}$ with $A\in \mathrm{SO}(2n)$. The stabiliser of the $\mathrm{SO}(2n)$-action is again $\mathrm{U}(n)$, so the orbit is given by $\mathrm{SO}(2n)/ \mathrm{U}(n)$, which amounts to all possible choices of orientation preserving almost complex structures compatible with the given inner product on $\bR^{2n}$. \par
From a slightly different point of view this can be formulated as follows: The inclusion $\mathrm{U}(n) \to \mathrm{SO}(2n)$ induces a fibre map $\pi \, : \,\mathcal{O}(\xi) \cong \mathrm{SO}(2n,1)/\mathrm{U}(n) \to \mathrm{SO}(2n,1)/\mathrm{SO}(2n) \cong H^{2n}$ with fibre isomorphic to $\mathrm{SO}(2n)/\mathrm{U}(n)$; the adjoint orbit $\mathcal{O}(\xi_0)$ fibres over hyperbolic space with the fibre over a point $x \in H^{2n}$ consisting of all almost complex structures compatible with the hyperbolic metric on $H^{2n}$ at $x$. This gives the identification $Z_{2n} \cong \mathcal{O}(\xi_0)$. From here on, $Z_{2n}$ will be used to denote the adjoint orbit $\mathcal{O}(\xi_0)$ , the corresponding coadjoint orbit, the homogenous space $\mathrm{SO}(2n,1)/\mathrm{U}(n)$ and the total space of the twistor fibration $\pi \, : \, Z_{2n} \to H^{2n}$.    \par
If $(x,J) \in Z_{2n}$ with $x \in H^{2n}$ and $J$ in the fibre over $x$, the tangent space at $(x,J)$ is
\begin{equation*}
 T_{(x,J)}Z_{2n} \cong \Lambda^2 (\bC^n)^* \oplus (\bC^n) \, ,  \,\bC^n = (\bR^{2n},J) \,.
\end{equation*}
We endow $Z_{2n}$ with an almost complex structure $I$ by demanding this identification to be $\bC$-linear. The resulting almost complex structure is the ``Eells--Salamon" structure of the twistor space $Z_{2n} \to H^{2n}$. (cf. \cite{EellsSalamon}). \par
As a coadjoint orbit, $Z_{2n}$ has already been endowed with $\mathrm{SO}(2n,1)$-invariant symplectic form $\omega$. It follows from homogeneous space description that $\mathrm{SO}(2n,1)$-invariant forms on $Z_{2n}$ are in one-to-one correspondence with $\mathrm{U}(n)$-invariant forms on $T_{\xi_0}Z_{2n} \cong \Lambda^{2}(\bC^n)^* \oplus \bC^n$. In the following, we will show that the space of closed $\mathrm{U}(n)$-invariant real 2-forms on $\Lambda^{2}(\bC^n)^* \oplus \bC^n$ is one-dimensional. \par

\begin{lemma} 
For $n>1$ the space of $\mathrm{U}(n)$-invariant real 2-forms on $\bC^n \oplus \, \Lambda^2 (\bC^n)^*$ is two-dimensional. Invariant forms are linear combinations of the standard hermitian forms on $P:= \bC^n$ and on $Q := \Lambda^2 (\bC^n)^*$. If $n=1$, $\mathrm{U}(n)$-invariant real 2-forms are multiples of the standard hermitian form on $P$.
\end{lemma}

\begin{proof}
Let $\Omega$ be an real invariant 2-form on $P \oplus Q$. In the obvious notation $\Omega$ can be written as
\begin{equation*}
 \Omega = \left( \begin{array}{c|c} \Omega_{P\times P} & \Omega_{Q\times P} \\ \hline
 \Omega_{P\times Q} &\Omega_{Q\times Q} \end{array}\right) \,.
\end{equation*} 
As $\mathrm{U}(n)$ is compact, $P$ and $Q$ are equivalent to their dual representations. Furthermore, if $n>1$, $P$ and $Q$ are irreducible and inequivalent, so by Schur's lemma we have $\Omega_{Q\times P} = 0$, $\Omega_{P\times Q} = 0$ and $\Omega_{P\times P} = \lambda_1 \Omega_1$ and $\Omega_{Q\times Q}=\lambda_2 \Omega_2$, where $\Omega_1$ and $\Omega_2$ are the standard hermitian forms  on  $P$ and $Q$ respectively. Since $\Omega$ is real, so are $\lambda_1, \lambda_2$.   \par
If $n=1$, then $Q=0$ and a degenerate version of the above argument shows that $\Omega = \lambda_1 \Omega_1$ for $\lambda_1 \in \bR$. 
\end{proof}

\begin{corollary}
The ``Eells-Salamon" almost complex structure $I$ on $Z_{2n}$ is compatible with the symplectic structure $\omega$. 
\end{corollary} 

\begin{remark}
If $n=2$, this is a special case of Theorem 4.4 in \cite{FinePanov2}.
\end{remark}

\begin{proof}
At $\xi_0$, $I$ is given by the $\mathrm{U}(n)$-invariant standard complex structure of  $\Lambda^2(\bC^n)^* \oplus \bC^n$, which is compatible with the standard hermitian forms on $\Lambda^2(\bC^n)^*$ and $\bC^n$. As $\omega$ is a linear combination of the latter two, $I$ and $\omega$ are compatible at $\xi_0$. By $\mathrm{SO}(2n,1)$-invariance of $\omega$ and $I$, the compatibility is global.  
\end{proof} \par
  					
\begin{proposition}  \label{Prop:1d}
The space of closed $\mathrm{SO}(2n,1)$-invariant real 2-forms on $Z_{2n}$ is one-dimensional consisting of real multiples of the standard symplectic form $\omega$ on the adjoint orbit $Z_{2n}$. 
 \end{proposition}
 
\begin{proof} 
Standard theory endows the adjoint orbit $Z_{2n} = \mathcal{O}(\xi_0)$ with a $\mathrm{SO}(2n,1)$-invariant symplectic form $\omega$ and a moment map $\mu \, : \, Z_{2n} \to \mathfrak{so}(2n,1)$ which is the inclusion of the adjoint orbit.  Let $\omega'$ be another $\mathrm{SO}(2n,1)$-invariant symplectic form on $Z_{2n}$ and assume is not a real multiple of $\omega$. As $\mathrm{SO}(2n,1)$ is semi-simple, its symplectic action on the simply connected space $(Z_{2n},\omega')$ admits a moment map $\mu' \, : \, Z_n \to \mathfrak{so}(2n,1)$ whose image is the adjoint orbit of $\xi_0':= \mu'(\xi_0)$. The elements $\xi_0$ and $\xi_0'$ are linearly independent in $\mathfrak{so}(2n,1)$ for $\xi_0' = \lambda \xi_0$ would imply $\mu' = \lambda \mu$ and hence $\omega' = \lambda \omega$. Their span is a two-dimensional subspace of $\mathfrak{so}(2n,1)$ on which the isotropy group $\mathrm{U}(n)$ of $\xi_0$ acts trivially, but the space of all elements in $\mathfrak{so}(2n,1)$ on which $\mathrm{U}(n)$ acts trivially is one-dimensional, consisting of imaginary multiples of the identity matrix in $\mathfrak{u}(n) \subset \mathfrak{so}(2n,1)$ (cf. Lemma~\ref{Lem:Decomp}). This is a contradiction, so $\omega'$ is a real multiple of $\omega$. \par
In terms of $\mathrm{U}(n)$-invariant real 2-forms $\Omega = \lambda_1 \Omega_1 + \lambda_2 \Omega_2$ on $\bC^n \oplus \Lambda^2(\bC^n)^*$ for $n>1$, this means that closedness imposes a fixed ratio between $\lambda_1$ and $\lambda_2$. In particular, neither $\Omega_1$ nor $\Omega_2$ can be closed and the only closed invariant real 2-forms are real multiples of $\omega$. The case $n=1$ is trivial. 
\end{proof}

\subsection{Symplectic Curvature Flow on $(Z_{2n},\omega)$} 
  
In order to run SCF on $Z_{2n}$ with the $\mathrm{SO}(2n,1)$-invariant almost K\"ahler structure $(\omega,I)$ serving as initial data, the Chern--Ricci curvature and the $I$-anti-invariant part of the Riemann--Ricci tensor need to be determined. \par

The Riemann--Ricci tensor $\Ric$ is determined by a $\mathrm{U}(n)$-invariant metric, so it is itself invariant. This is enough to see that $\Ric$ has to be $I$-invariant: 
If multiplication by $i$ on $\Lambda^2 (\bC^n)^* \oplus \bC^n$ were represented by an element in $\mathrm{U}(n)$ this would be immediate. This is not the case, but there is an easy work-around. Set $z^{ij} := z^i \wedge z^j$ and consider the basis $(z^{ij}, z_k)$ of $\Lambda^2 (\bC^n)^* \oplus \bC^n$. Since $\Ric$ is symmetric bilinear, it is determined by its values on $((z^{ij},0),(z^{i'j'},0))$, $((0,z_k),(0,z_{k'}))$ and $((z^{ij},0),(0,z_k))$. 
For each of these pairs of arguments there exists an element $\operatorname{diag}(e^{i\lambda_1},...,e^{i\lambda_n}) \in \bT^n \subset \mathrm{U}(n)$  acting by multiplication by $i$ on both, so $\Ric$ has to be $I$-invariant. Consequently, $\mathcal{R} = [Rc,I] = 0$. \par

The Chern--Ricci tensor $P$ is a closed $\mathrm{SO}(2n,1)$-invariant 2-form, so by Proposition~\ref{Prop:1d}, $P$ is a multiple of $\omega$. In \cite{FinePanov1} J. Fine and D. Panov  determined the first Chern class of $Z_{2n}$: $c_1(Z_{2n})  =  (n-2) [\omega]$. As $(1/4\pi)P$ represents the first Chern class ($P$ is $2i$ times the curvature of the Chern connection on the anti-canonical bundle, $i/2\pi$ times which represents the first Chern class), we have $P = (n-2)\pi/4  \cdot \omega$. In particular, $P$ has no $(2,0)$ and $(0,2)$ parts. \par

With this result, SCF for $(Z_{2n},\omega)$ becomes 
\begin{equation*}
 \partial_t \omega(t) = \frac{\pi}{2}(2-n) \cdot \omega(0) \, , \qquad \partial_t I = 0 \,.
\end{equation*}
It is manifest that SCF collapses $(Z_{2n},\omega)$ in finite time if $n>2$, expands it if $n=1$ and leaves the almost K\"ahler structure unchanged if $n=2$.

\subsection{Non-K\"ahler Quotients} 

The symplectic form $\omega$ and the almost complex structure $I$ on $Z_{2n}$ are $\mathrm{SO}(2n,1)$-invariant, so the almost K\"ahler structure will descent to quotients of $Z_{2n}$ by subgroups $\Gamma \subset \mathrm{SO}(2n,1)$. In the adjoint orbit description of $Z_{2n}$, $\Gamma$ acts by conjugation. Viewing $Z_{2n} \cong \mathrm{SO}(2n,1)/\mathrm{U}(n)$ ($\mathrm{U}(n)$ acting from the right), this corresponds to $\Gamma$ acting by left multiplication, so the actions of $\mathrm{U}(n)$ and $\Gamma$ on $\mathrm{SO}(2n,1)$ commute. \par

If one chooses $\Gamma \subset \mathrm{SO}(2n,1)$ to be the fundamental group of a compact hyperbolic manifold $M$ of dimension $2n$, one obtains two quotients: $\Gamma \backslash H^{2n} \cong \Gamma \backslash \mathrm{SO}(2n,1) / \mathrm{SO}(2n) =: M$ and
 $\Gamma \backslash Z_{2n} \cong \Gamma \backslash \mathrm{SO}(2n,1) / \mathrm{U}(2n)$. The action of $\Gamma$ on $Z_{2n}$ and $H^{2n}$ commutes with the projection $\pi  :  Z_{2n} \to H^{2n}$, so $\Gamma \backslash Z_{2n}$ fibres over $M$ with fibre $\mathrm{SO}(2n)/\mathrm{U}(n)$. This shows that $\Gamma \backslash Z_{2n}$ is a fibre bundle with compact base and fibre, so it is itself compact. Furthermore, the fibre $\mathrm{SO}(2n)/\mathrm{U}(n)$ is connected and simply connected, so $\Gamma \backslash Z_{2n}$ and $M$ have isomorphic fundamental groups  $\pi_1(\Gamma \backslash Z_{2n}) \cong \pi_1(M) \cong \Gamma$, but compact K\"ahler manifolds cannot have fundamental group isomorphic to that of a compact hyperbolic manifold in dimension greater than $2$ (see e.g. \cite{Toledo}). Hence, $\Gamma \backslash Z_{2n}$ cannot be K\"ahler if $n>1$.

\section{SCF on left-invariant structures on some four- and six-dimensional nilmanifolds \label{Sec:SCFLeftInvar}}

In the case of left-invariant almost K\"ahler structures on a nilpotent Lie group, SCF reduces to an ODE on the corresponding nilpotent Lie algebra. Moreover, if the structure coefficients of a connected, simply connected Lie group's Lie algebra can be chosen rational, the Lie group admits cocompact lattices (Theorem 7 in \cite{Malcev}). As non-abelian nilpotent Lie algebras are never formal (cf. \cite{Hasegawa}), taking quotients by such lattices results in compact non-K\"ahler manifolds on which we can hope to explicitly solve SCF. \par

This section presents such explicit solutions for SCF on three different nilalgebras. For the computations involved, the expression for the Chern--Ricci form provided in the following lemma is useful.

\begin{lemma} \label{Lem:A}
Let $(M,g,J,\omega)$ be an almost K\"ahler manifold. Denote by $A$ the connection 1-form of the Levi--Civita connection in a local complex frame (a local frame in which $J$ is constant). In that frame the Chern--Ricci form has the following expression:
\begin{equation*}
 P =  d \tr (AJ) \,.
\end{equation*}
\end{lemma}

\begin{proof}
The Chern connection on an almost hermitian manifold $(M,g,J,\omega)$ is the unique connection $\nabla$ with respect to which $g$ and $J$ are parallel and whose torsion has vanishing $(1,1)$-part. In the almost K\"ahler case it is given by  $\nabla_X Y = D_X Y - \frac{1}{2}J(D_X J)Y$, where $D$ denotes the Levi--Civita connection. If $A$ and $C$ are the connection 1-forms of the Levi--Civita connection and the Chern connection in a local complex frame, then the formula for the Chern connection can be expressed as 
\begin{equation*}
 C = A - \frac{1}{2}J(DJ) = A - \frac{1}{2}J[A,J] = \frac{1}{2}(A - JAJ)
\end{equation*}

Denote by $F$ the full curvature tensor of $\nabla$ given by the endomorphism-valued 2-form $F = dC + C \wedge C$. The Chern--Ricci tensor is derived from $F$ via $P_{kl} = \omega^{ij}F_{ijkl}$, where $ij$ are the endomorphism indices ($i$ lowered via the metric) and $kl$ the form indices. Omitting the form indices, a brief calculation yields
\begin{equation*}
 P = \omega^{st}F_{st}= \omega_{it}F^{it} = g_{jt}J^j_i F^{it} = J^j_i F^i_j = \tr(JF) \; .
\end{equation*}

Application of the 2-form $C\wedge C$ to a pair of tangent vectors $u,v$ gives $(C \wedge C) (u,v) = [C_u,C_v]$. The fact that for two endomorphisms $A,B$ one has $\tr(AB) = \tr(BA)$ in conjunction with $[J,C]=0$ implies
\begin{equation*}
 (\tr J C\wedge C)(u,v) = \tr (J[C_u,C_v]) =  \tr (J C_u C_v ) - \tr(J C_v C_u) = 0  \,, 
\end{equation*}
i.e. $\tr (J C \wedge C)=0$. Since $dJ = 0$, the remaining contribution to the Chern--Ricci form is 
\begin{equation*}
 P = \tr (JdC) = d\tr (JC) = \frac{1}{2}d \, [\tr (JA) + \tr (AJ)] = d \tr (AJ)   \, 
\end{equation*}
as claimed.
\end{proof}

We want to apply this result to left-invariant almost K\"ahler structures on Lie groups, in which case left-invariant frames are complex frames. With the help of the next lemma, the Chern--Ricci form $P$ can be expressed directly in terms of the Lie algebra and the almost complex structure.  

\begin{lemma} \label{Lem:B}
Let $G$ be a Lie group and $(g,J,\omega)$ a left-invariant almost K\"ahler structure. If $A$ is the connection 1-form of the Levi--Civita connection in a left-invariant fame, then for any left-invariant vector field $Z \in \mathfrak{g}$ one has
\begin{equation*}
 \tr A_Z J = \frac{1}{2} \tr( \ad_Z \circ J + J \circ \ad_Z) + \tr \ad_{JZ} \, .
\end{equation*}
\end{lemma}

\begin{proof}
Viewing the almost K\"ahler structure $(g,J, \omega)$ as algebraic data on the Lie algebra $\mathfrak{g}$ of $G$, the condition that the alternating bilinear form $\omega = g(J\cdot,\cdot)$ be closed means that $0=d\omega(X,Y,Z) = - \omega([X,Y],Z) + \omega([X,Z],Y) - \omega([Y,Z],X)$ for any $X,Y,Z \in \mathfrak{g}$. \par

Now let $(e_i)$ be an orthonormal left-invariant frame of $\mathfrak{g}$. Using the Kozul-formula 
\begin{equation*}
 2g(Y,A_{Z}X) = g([Z,X],Y) - g([Z,Y],X) - g([X,Y],Z) \,,
\end{equation*}
the desired result follows from a straightforward computation:
\begin{eqnarray*}
 2 \tr A_Z J &=& 2 \sum_i g(e_i, A_Z J e_i) \\
 &=& \sum_i g([Z,Je_i],e_i) - g([Z,e_i],Je_i) - g([Je_i,e_i],Z) \\
 &=& \sum_i g(\ad_Z \circ J (e_i),e_i) + g(J \circ \ad_Z (e_i), e_i) - g([Je_i,e_i],Z) \\
 &=& \tr (\ad_Z \circ J + J \circ \ad_Z) - \sum_i g([Je_i,e_i],Z) \,.
\end{eqnarray*}
By the closedness of $\omega$, the second term on the right hand side can be expressed as $2 \tr \ad_{JZ}$:
\begin{eqnarray*}
 - \sum_i g([Je_i,e_i],Z) &=& -\sum_i \omega([Je_i,e_i],JZ)  \\
 &=&  \sum_i \omega([Je_i, JZ], e_i) - \omega([e_i,JZ], Je_i)  \\
 &=& \sum_i g(J[Je_i,JZ], e_i) - g(J[e_i, JZ], Je_i) \\
 &=& \sum_i g( \ad_{JZ}J e_i, J e_i) + g(\ad_{JZ} e_i, e_i) \\
 &=& 2 \tr \ad_{J Z}  \,.
\end{eqnarray*} 
\end{proof}

Since for any left-invariant 1-form $\theta \in \mathfrak{g}^*$ and $X,Y \in \mathfrak{g}$ the relation $d \theta (X,Y) = -\theta([X,Y])$ holds, Lemmas~\ref{Lem:A} and \ref{Lem:B} combine to
\begin{equation*}
 P(X,Y) = (d\tr AJ)(X,Y) =  - \tr A_{[X,Y]}J = - \frac{1}{2}\tr (\ad_{[X,Y]} \circ J + J \circ \ad_{[X,Y]}) - \tr \ad_{J[X,Y]} .
\end{equation*}

This has a very useful consequence: 

\begin{proposition}[L. Vezzoni]\footnote[1]{This was brought to the author's attention by Luigi Vezzoni in a private conversation. His proof will be published in \emph{A note on canonical 
Ricci forms on 2-step nilmanifolds} in 
Proc. AMS.}  \label{Lem:ChernRicFlat}
All left-invariant almost K\"ahler structures on two-step nilpotent Lie groups are Chern--Ricci flat. 
\end{proposition} 

\begin{proof}
Let $G$ two-step nilpotent Lie group with fixed almost K\"ahler structure and $\mathfrak{g}$ the Lie algebra of $G$. The assumption that $G$ is two-step then means that $[[\mathfrak{g},\mathfrak{g}],\mathfrak{g}]=0$, i.e. for any $X,Y \in \mathfrak{g}$ we have $\ad_{[X,Y]} =0$, so
\begin{equation*}
 P(X,Y) = - \frac{1}{2}\tr (\ad_{[X,Y]} \circ J + J \circ \ad_{[X,Y]}) - \tr \ad_{J[X,Y]} = - \tr \ad_{J[X,Y]} \,.
\end{equation*} \
Now choose an orthonormal basis $(e_i)$ of $\mathfrak{g}$ with the property that each $e_j$ lies either in $[\mathfrak{g},\mathfrak{g}]$ or in $[\mathfrak{g},\mathfrak{g}]^\perp$. Then the summands of
\begin{equation*}
 \tr \ad_{Z} = \sum_i g([Z,e_i],e_i)  
\end{equation*}
vanish since either $e_i \in [\mathfrak{g},\mathfrak{g}]$ and therefore $[\cdot,e_i] = 0$ (two-step property) or $e_i \in [\mathfrak{g},\mathfrak{g}]^\perp$ and $g([\cdot,e_i],e_i) =0$. 
\end{proof}

It should be noted that on manifolds with Chern--Ricci flat almost K\"ahler structure symplectic curvature flow reduces to anti-complexified Ricci flow introduced by H.V. Le and G. Wang in \cite{LeWang}. 

\subsection{Kodaira--Thurston manifold}
The simplest example of a symplectic nilmanifold is the Kodaira--Thurston manifold which can be realised as a product of $S^1$ and the quotient of the three-dimensional Heisenberg group 
\begin{equation*}
 H_3 := \left \{ \begin{pmatrix} 1& x_1 & x_3 \\ 0 & 1 & x_2 \\ 0& 0 & 1    \end{pmatrix}   \Bigg| x_1,x_2,x_3 \in \bR   \right \}
\end{equation*}
by the obvious integral lattice $\Gamma := H_3 \cap \mathrm{Gl}(3,\bZ) \subset H_3$. Topologically, the Kodaira--Thurston manifold is a $S^1$-bundle over a three-torus where the fibers are given by the central direction in $H_3$ and the base by the two unpreferred directions in $H_3$ and the additional $S^1$ direction. \par The Lie algebra $\mathfrak{h}_3 \oplus \bR$ of $H_3 \times \bR$ is given by generators $e_1,...,e_4$ with $[e_1,e_2] = e_3$ as the only nontrivial Lie bracket. Equivalently, of the dual basis vectors $e^1,...,e^4$ of $(\mathfrak{h}_3 \oplus \bR)^*$ the only one whose corresponding left invariant 1-form is not closed is $e^3$ with $de^3 = -e^1 \wedge e^2$. By Proposition~\ref{Lem:ChernRicFlat}, any left-invariant almost K\"ahler structure defined on the 2-step nilalgebra $\mathfrak{h}_3 \oplus \bR$ is Chern--Ricci flat and SCF leaves the symplectic form of an initial left-invariant almost K\"ahler structure unchanged. The evolution equation then is just $\partial_t J = \mathcal{R}$ or equivalently $\partial_t g= -\Ric + \Ric(J\cdot,J\cdot)$ (the equivalence can be seen by observing $0 = \partial_t \omega = (\partial_t g)(J\cdot,\cdot) + g(\partial_t J \cdot, \cdot)$). \par

Consider the following two parameter family of almost K\"ahler structures (matrices interpreted in the $e_i/e^j$ bases) with positive parameters $\alpha, \beta$:
\begin{equation*}
 \omega = e^1\wedge e^3 - e^2 \wedge e^4 \, ,
\end{equation*}
\begin{equation*}
 J = ({J}^i_{\,j}) =
 \begin{pmatrix}
  0 & 0 & -\alpha& 0 \\
  0 & 0 & 0 & \beta \\
  \alpha^{-1} & 0 & 0 & 0 \\
  0& -\beta^{-1} & 0 & 0 \\
 \end{pmatrix} \, \quad
 g=(g_{ij}) =  
 \begin{pmatrix}
  \alpha^{-1} & 0 &  0& 0 \\
  0 & \beta^{-1} & 0 & 0  \\
  0 & 0 & \alpha & 0 \\
  0&  0 & 0 & \beta \\
 \end{pmatrix} \, .
\end{equation*} \par 
Computing the connection $1$-form $A$ of the Levi--Civita connection $D$ via the Koszul formula gives
\begin{equation*}
 A= \frac{1}{2}
 \begin{pmatrix}
  0 & \alpha^2 e^3 & \alpha^2 e^2 & 0 \\
  -\alpha \beta e^3 & 0 & - \alpha \beta e^1 & 0 \\
  -e^2 & e^1 &  0 &0 \\
  0 & 0 & 0 & 0
 \end{pmatrix} \, .
\end{equation*}
The Ricci Tensor is then given by $\Ric_{jk}=R_{\;j \, kl}^{k} $, where $R = dA + A \wedge A$ is full the Riemann curvature tensor:
\begin{equation*} \Ric = \frac{1}{2}
 \begin{pmatrix}
  -\alpha \beta & 0 &  0& 0 \\
  0 & -\alpha^2 & 0 & 0 \\
  0 & 0 &  \alpha^3 \beta &0 \\
  0 & 0 & 0 & 0
 \end{pmatrix} \, .
\end{equation*}
Finally, SCF is determined by  $\partial_t g = - \Ric + \Ric(J\cdot,J\cdot)$, so 
\begin{equation*}
 \partial_t 
 \begin{pmatrix}
  \alpha^{-1} & 0 &  0& 0 \\
  0 & \beta^{-1} & 0 & 0  \\
  0 & 0 & \alpha & 0 \\
  0&  0 & 0 & \beta \\
 \end{pmatrix} 
 = \frac{1}{2} 
 \begin{pmatrix}
  2 \alpha \beta & 0 & 0 & 0 \\
  0 & \alpha^2 &0 & 0 \\
  0& 0& -2 \alpha^3 \beta & 0 \\
  0 & 0 & 0 & - \alpha^2 \beta^2
 \end{pmatrix} \,. 
\end{equation*} \par
The resulting equations $\partial_t \alpha = - \alpha^3 \beta, \; \partial_t \beta = -\frac{1}{2} \alpha^2 \beta^2$ can easily be integrated observing that $\partial_t(\alpha^{-\frac{2}{3}}\beta^{\frac{4}{3}}) = 0$. The general solution for initial values $\alpha(0) = \alpha_0, \, \beta(0) = \beta_0$ is given by
\begin{equation*}
 \alpha(t) = \alpha_0 \left(1+ \frac{5}{2} \alpha_0^2 \beta_0 \cdot t \right)^{-\frac{2}{5}}\! \! , \quad  \beta(t) = \beta_0 \left(1+ \frac{5}{2} \alpha_0^2 \beta_0 \cdot t \right)^{-\frac{1}{5}} \! \! .
\end{equation*}
Geometrically, this means that symplectic curvature flow shrinks the central directions of $H_3 \times \bR$ while expanding the unpreferred directions at inverse rates. The shrinking of the central direction in $H_3$ and that of $\bR$ occur at different rates, the former collapsing faster than the latter. The corresponding unequal expansion of the unpreferred directions $e_1, e_2$ is due to the choice of symplectic form which couples $e_1, e_3$ and $e_2, e_4$.  \par
Two quantities whose behaviour under SCF might be of interest are the (pointwise) norms of the Nijenhuis tensor and the Riemann curvature tensor.  One finds 
\begin{equation*}
 ||N||^2 = 8 \alpha^2 \beta =  \frac{8 \alpha_0^2 \beta_0}{1+ \frac{5}{2}\alpha_0^2 \beta_0 \cdot t}   \,  , \quad ||R||^2 = \frac{11}{4} \alpha^4 \beta^2 = \frac{11}{4}   \frac{\alpha_0^4 \beta_0^2}{(1+ \frac{5}{2}\alpha_0^2 \beta_0 \cdot t)^2} \,.
\end{equation*} \par
Symplectic curvature flow on the Kodaira--Thurston manifold has also been considered in \cite{LeWang} as an instance of anti-complexified Ricci flow, but it appears the example therein is faulty (e.g. the given solution doesn't satisfy the initial conditions) and we felt it would be worth including our own computation.

\subsection{Sum of two Heisenberg algebras}

The computation is similar for the product of two Heisenberg groups. The generators  $e_1,...,e_6$ of its Lie algebra can be chosen such that $[e_1,e_2] = e_5$ and $[e_3,e_4]=e_6$. The Lie algebra $\mathfrak{h}_3 \oplus \mathfrak{h}_3$ is two step, so Lemma~\ref{Lem:ChernRicFlat} can be applied. In the following, all matrices are with respect to the $e_i/e^j$ bases.  \par

Consider the following three parameter family of almost Kähler structures for positive paramters $\alpha, \beta, \gamma$:
\begin{eqnarray*}
 \omega &=& e^1 \wedge e^5 + e^2 \wedge e^4 + e^3 \wedge e^6 \, , \\
 g &=& \alpha^{-1} e^1 \otimes e^1 + \beta^{-1} e^2 \otimes e^2 + \gamma^{-1} e^3 \otimes e^3 + \beta e^4 \otimes e^4 + \alpha e^5 \otimes e^5 + \gamma e^6 \otimes e^6 \, , \\
 J &=& 
 \begin{pmatrix}
  0_3 & \begin{array}{ccc} 0& - \alpha & 0 \\ - \beta & 0 & 0 \\ 0 & 0 & -\gamma \end{array} \\
  \begin{array}{ccc} 0& \beta^{-1} & 0 \\ \alpha^{-1} & 0 & 0 \\ 0 & 0 & \gamma^{-1} \end{array} & 0_3
 \end{pmatrix} \,.
\end{eqnarray*}

As in the case of $\mathfrak{h}_3 \oplus \bR$, the flow equation of SCF can be written as $\partial_t g = -\Ric + \Ric(J\cdot, J \cdot)$, where the Ricci tensor is computed from the connection 1-form of the Levi--Civita connection via the full Riemann curvature tensor. We obtain
\begin{equation*}
 \Ric = \frac{1}{2}
 \begin{pmatrix}
  -\alpha \beta & 0 & 0 & 0 & 0 & 0 \\
  0 & - \alpha^2 & 0 & 0 & 0 & 0 \\
  0 & 0& - \gamma \beta^{-1} & 0 & 0 & 0 \\
  0 & 0& 0 & - \gamma^2 & 0 & 0 \\
  0 & 0& 0 & 0 & \alpha^3 \beta & 0 \\
  0 & 0& 0 & 0 & 0 & \gamma^3 \beta^{-1} 
 \end{pmatrix}
\end{equation*} 
and
\begin{equation*}
 \partial_t g = \frac{1}{2}
 \begin{pmatrix}
  2\alpha \beta & 0 & 0 & 0 & 0 & 0 \\
  0 & -\gamma^2 \beta^{-2} + \alpha^2 & 0 & 0 & 0 & 0 \\
  0 & 0&  2 \gamma \beta^{-1} & 0 & 0 & 0 \\
  0 & 0& 0 & - \alpha^2 \beta^2 + \gamma^2 & 0 & 0 \\
  0 & 0& 0 & 0 & -\alpha^3 \beta & 0 \\
  0 & 0& 0 & 0 & 0 & -\gamma^3 \beta^{-1} 
\end{pmatrix} \,.
\end{equation*} \par 
The resulting equations for $\alpha, \beta, \gamma$ are 

\begin{equation*}
 \partial_t \alpha = - \alpha^3 \beta, \qquad   \partial_t \beta = - \frac{1}{2} \alpha^2 \beta^2 + \frac{1}{2}\gamma^2, \qquad \partial_t \gamma = - \gamma^3 \beta^{-1}
\end{equation*}
with initial conditions $\alpha(0) = \alpha_0, \beta(0) = \beta_0, \gamma(0) = \gamma_0$. The equation for $\partial_t \beta$ can be rewritten as $2 \partial_t \log \beta = \partial_t \log \alpha/\gamma$, so $\beta/\beta_0 = (\alpha/\alpha_0)^{\frac{1}{2}}(\gamma/\gamma_0)^{-\frac{1}{2}}$. With this expression for $\beta$ the other two equations read

\begin{equation*}
 \partial_t \alpha = - L \alpha^{\frac{7}{2}} \gamma^{-\frac{1}{2}}, \qquad \partial_t \gamma = - L^{-1}\gamma^{\frac{7}{2}} \alpha^{-\frac{1}{2}} \,,
\end{equation*}
where $L = \beta_0 (\gamma_0 / \alpha_0)^{\frac{1}{2}}$. \par

In the case $\beta_0= \gamma_0 / \alpha_0$, these equations can be integrated without much difficulty and the solutions are 
\begin{equation*}
 \alpha(t) = \alpha_0(1+ 2 \alpha_0 \gamma_0 \cdot t)^{-\frac{1}{2}}, \qquad \beta(t) = \beta_0, \qquad \gamma(t) = \gamma_0(1+ 2\alpha_0 \gamma_0 \cdot t)^{-\frac{1}{2}} \,.
\end{equation*}

As on the Kodaira--Thurston manifold, symplectic curvature flow shrinks the central directions in each of the copies of $H_3$ and expands the base direction coupled to the central ones by the symplectic form at the inverse rate. \par

For general initial conditions, integration of the equations for $\alpha$ and $\gamma$ becomes more difficult. One may substitute $\xi := L^{-1}\alpha^{-3}, \, \eta := L \gamma^{-3}$. Then $\partial_t \xi = \partial_t \eta$, so $\xi = \eta + c$, where $c = L^{-1}\alpha_0^{-3} - L \gamma_0^{-3}$. The case $c=0$ corresponds exactly to the ``easy" case considered previously. The equation for $\eta$ reads
\begin{equation*}
 \partial_t \eta = 3 \eta^{\frac{1}{6}}(\eta +c)^{\frac{1}{6}} \, .
\end{equation*}
Integration is possible in terms of hypergeometric series, but the author has not pursued the analysis. Qualitatively, the behaviour is expected to be similar to the easy case with the central directions collapsing, the two base directions coupled to the central directions by the symplectic form expanding at inverse rates and the remaining two base directions coupled to each other tending to a finite scale. \par

The pointwise norms of the Nijenhuis and Riemann tensors are given by 
\begin{equation*}
 ||N||^2 = 8 (\alpha^2 \beta + \gamma^2 \beta^{-1})  \,  , \quad ||R||^2 = \frac{11}{4}(\alpha^4 \beta^2 + \gamma^4 \beta^{-2} ) \,.
\end{equation*}
In the case where $\beta_0 = \gamma_0/\alpha_0$, these reduce to 
\begin{equation*}
 ||N||^2 = 16 \alpha \gamma =  16 \frac{ \alpha_0 \gamma_0}{1+  2 \alpha_0 \gamma_0 \cdot t}   \,  , \quad ||R||^2 = \frac{11}{2} \frac{\alpha_0^2\gamma_0^2}{(1+ 2 \alpha_0 \gamma_0 \cdot t)^2}  \,.
\end{equation*}

\subsection{The nilalgebra $\mathfrak{n}_4$}

The situation changes for the nilalgebra $\mathfrak{n_4}$ with generators $e_1,...,e_4$ and $[e_1,e_2]=e_3$ and $[e_2,e_3]=e_4$ as the only nonvanishing commutators. This nilalgebra is three-step, so Lemma~\ref{Lem:ChernRicFlat} doesn't hold and SCF turns out to evolves both $\omega$ and $J$ non-trivlially. \par

The initial almost K\"ahler structure considered is $\omega_0 = e^1\wedge e^3 + e^2 \wedge e^4$ and $J_0 = e_3 \otimes e^1 + e_4 \otimes e^2 - e_1 \otimes e^3 - e_2 \otimes e^4$. $e^i \in \mathfrak{n}_4^*$.  The symplectic form $\omega_0$ is closed since $de^1 =de^2 = 0$ and $de^3 = -e^1 \wedge e^2$ and $de^4 = - e^2 \wedge e^3$. \par
To run symplectic curvature flow, $\partial_t (\omega,J)$ needs to be known on a sufficiently large space of almost K\"ahler structures on $\mathfrak{n}_4$. For computational convenience the following familiy of almost K\"ahler structures was chosen:
\begin{eqnarray*}
 \omega = e^1\wedge e^3 + e^2 \wedge e^4 + \gamma e^1 \wedge e^2, \quad
 J = 
 \begin{pmatrix} 
  0 & a' & b' & 0 \\
  a & 0 & 0 & c' \\
  b & 0 & 0 & d' \\
  0 & c & d & 0 
 \end{pmatrix} \, .
\end{eqnarray*}
The matrix $J$ is to be understood as an endomorphism of $\mathfrak{g}$ in the $e_i/e^j$ bases. The fact that $J$ is an almost complex structure imposes algebraic relations on $a,b,c,d,a',b',c',d'$:
\begin{eqnarray*}
 aa' + bb' = -1,&&  ac+bd=0 \\
 aa' + cc' = -1,&& a'c' + b'd' = 0 \\
 bb' + dd' = -1,&& ab' + c'd =0 \\
 cc' + dd' =-1,&& a'b + cd' =0 \,.
\end{eqnarray*}
The equations on the right hand side are all equivalent in light of the ones on the left, of which only three are independent. Furthermore, the compatibility condition $\omega(J\cdot,J\cdot) = \omega$ fixes $\gamma$ by $b'\gamma = a' +d$, so the above defines a four-dimensional space of almost K\"ahler structures on $\mathfrak{n}_4$. \par

The metric associated to $\omega, J$ in the $e_i/e^j$ basis is given by
\begin{equation*}
 (g_{ij}) = 
 \begin{pmatrix}
  g_{11} & 0 & 0 & g_{14}  \\  0 & g_{22} & g_{23} & 0 \\ 0 & g_{23} & g_{33} & 0 \\ g_{14} & 0 & 0 & g_{44} 
 \end{pmatrix} 
 = 
 \begin{pmatrix}
  b+\gamma a & 0 & 0 & -a  \\  0 & c-\gamma a' & -a' & 0 \\ 0 & -a' & -b' & 0 \\ -a & 0 & 0 & -c' 
 \end{pmatrix} \, .
\end{equation*} \par

Infinitesimal changes of these almost K\"ahler structures under SCF are determined by the Chern--Ricci form and the Ricci curvature (more precisely, the $(2,0)+(0,2)$-part of the Ricci curvature, since $2g^{-1}\Ric^{(2,0)+(0,2)}=J\mathcal{R}$).  \par To compute them, let $D$ denote the Levi--Civita connection of the left-invariant metric $g$. Its connection 1-form $A$ in the $e_i/e^j$ is the element of $\operatorname{End}(\mathfrak{n}_4) \otimes \mathfrak{n}_4^*$ given by
\begin{equation*}
 2g(e_k,D_{e_j}e_i) = 2g(e_k, A^l_{\, ij} e_l) = g([e_j,e_i],e_k) - g([e_j,e_k],e_i) - g([e_i,e_k],e_j)
\end{equation*}
or, more explicitly, by 
\begin{eqnarray*}
 2d_{14}d_{23}A &=& 
 \begin{pmatrix}
  0 & 0 & 0 & 0 \\
  0 & -g_{23}(g_{33} - g_{14})d_{14} & - g_{33}(g_{33} - g_{14})d_{14} & 0 \\
  0 & g_{22}(g_{33} - g_{14})d_{14} & g_{23}(g_{33} - g_{14})d_{14} &0 \\
  0 & 0 & 0 &0 
 \end{pmatrix}
 e^1 \\ &+&
 \begin{pmatrix}
  0 & 0 & 0 & 0 \\ 
  0 & g_{23}g_{44}d_{14} &  g_{33}g_{44}d_{14}& 0 \\
  0 & -g_{22}g_{44}d_{14} & -g_{23}g_{44}d_{14}&0 \\
  0 & 0 & 0 & 0 
 \end{pmatrix}
 e^4 \\ &+&
 \begin{pmatrix}
  0 & 2g_{23}g_{44}d_{23} & g_{33}g_{44}d_{23} & 0 \\
  -g_{23}(g_{33}-g_{14})d_{14} & 0 & 0 & g_{23}g_{44}d_{14} \\
  -(d_{23}+g_{22}g_{14}-g_{23}^2)d_{14} & 0 & 0 & -g_{22}g_{44}d_{14} \\
  0 & -2g_{23}g_{14}d_{23} & (d_{14}-g_{14}g_{33})d_{23} & 0 
 \end{pmatrix}
 e^2 \\ &+&
 \begin{pmatrix}
  0 & g_{33}g_{44}d_{23} & 0 & 0 \\
  -g_{33}(g_{33}-g_{14})d_{14} & 0 & 0 & g_{33}g_{44}d_{14} \\
  g_{23}(g_{33}-g_{14})d_{14} & 0 & 0 & -g_{23}g_{44} d_{14} \\
  0 & -(d_{14} + g_{14}g_{33})d_{23} & 0 & 0 
 \end{pmatrix}
 e^3 \,.
\end{eqnarray*}
Here $d_{14} = g_{11}g_{44}-g_{14}^2$ and $d_{23} = g_{22}g_{33}-g_{23}^2$. Observe $d_{14}d_{23}=\det g_{ij} = \det \omega_{ij} \cdot \det J$. For $\omega, J$ in the considered family, this is equal to $\det J=\det J_0 = 1$.  \par

With $A$ known, the Riemann curvature $F_D$  is then given by the endomorphism valued 2-form $A \wedge A + dA$. The Ricci curvature viewed as endomorphism of $\mathfrak{n}_4$ by means of $g$ turns to out to be 
\begin{equation*}
 Rc = g^{-1}\Ric = 
 \begin{pmatrix}
  -g_{33}^2 g_{44} & 0 & 0 & 0 \\ 
  0 & -g_{44}(g_{33}^2 +d_{14}) & 0 & 0 \\
  0 & 2g_{23}g_{33}g_{44} & g_{44}(g_{33}^2-d_{14}) & 0 \\
  g_{14}(g_{33}^2+d_{14}) & 0 & 0 & g_{44}d_{14} 
 \end{pmatrix} \, .
\end{equation*}
Computing the commutator $[Rc,J]$ and expressing the $g_{ij}$ in terms of entries of $J$  yields for $2\mathcal{R}$:
\begin{equation} \label{Eq:CommPart}
 \begin{pmatrix} \; \; \;\quad  0 \qquad \qquad \qquad \qquad \qquad \; a'c'(b'^2 - d_{14})               &   \!\!\!\!\!\!\!        b'c'(2b'^2 - d_{14})  \qquad \qquad \qquad \qquad 0 \\
  \!\!\!\!\!\!\!\!    ac'(b'^2 \!+\!2d_{14}) \qquad \qquad \qquad \qquad \qquad  0                     &      \;\;\; \quad 0  \quad \qquad \qquad \qquad \quad c'^2(b'^2 \! + \!2d_{14}) \\
  \!\!\!\!\!\!\!\!\!\!\!\!\!\!\!\! -2aa'b'c' \!-\! bc'(2b'^2\!\!-\!d_{14})\!+\!ad'(b'^2\!+\!d_{14}) \; \; \, 0 &     \qquad 0 \qquad \quad \! -\! 2ca'b'c' \! - \! d'c'(b'^2 \! \! - \! 2d_{14}) \\
  \qquad \; \; 0 \, - \!  aa'(b'^2\!\!+\!d_{14}) \!-\! cc'(b'^2 \!\!+\! 2 d_{14}) \!+\! 2d a' b' c'                          &    \!  \!\!\!\!\!\!\!  - ab'(b'^2\! \! +\!d_{14}) \!+\! d c'(b'^2 \! \!- \!2d_{14}) \qquad \; 0
 \end{pmatrix}
\end{equation} \par

The second quantity required too write out the SCF equations explicitly is the Chern--Ricci tensor $P$, for which a convenient expression was derived in Lemma~\ref{Lem:A}: 
\begin{equation*}
 P = \tr (J dA) \, .
\end{equation*}
With the $A$ given above it is $P = c' e^1 \wedge e^2$. Furthermore,
\begin{equation} \label{Eq:ChernRicPart}
 -2 g^{-1}P^{(2,0)+(0,2)} = 
 \begin{pmatrix}
  0 & -b'c' & 0 & 0 \\
  c'^2 & 0 & 0 & 0 \\
  d'c'+ac' & 0 & 0 & c'^2 \\
  0 &  b'd' + ab' & -b'c' & 0  
 \end{pmatrix}  \, .
\end{equation} \par

Along with the expression for $\mathcal{R}$ found in Equation~(\ref{Eq:CommPart}) this constitutes the evolution equation $\partial_t J = -2 g^{-1}P^{(2,0)+(0,2)} + \mathcal{R}$. Setting $y(t) = (1+5/2 \cdot t)^{1/5}$, the explicit solution to this ODE with the initial condition $J(0) = J_0$ is given by 

\begin{equation*}
 \begin{array}{cclcccl}
  a & \! = \! & y^{-1} - y^{-3} &\!\!, \quad & b & \! = \! &  2y^{-1} - y^{-3}\, \\
  c & \!=\!   &  2y-y^{-1}       &\!\!,\quad & d  & \!=\! & -y + y^{-1} \, \\
  a' \! &\!=\! & -y + y^{-1}      &\!\!, \quad & b' \!&\!=\!& -y^{-1}\, \\ 
  c' \!&\!=\!& -y^{-3}               &\!\!,\quad & d' \!&\!=\!& y^{-1} - y^{-3} \,.
 \end{array}
\end{equation*}\par 
For the evolution of $\omega$ according to $\partial_t \omega = -2 P$ with $\omega(0) = \omega_0$ one obtains
\begin{equation*}
 \omega(t) = e^1 \wedge e^3 + e^2 \wedge e^4 + 2(y^2 -1)e^1 \wedge e^2 
\end{equation*}	
and the metric evolves as
\begin{equation*} (g_{ij}) = g(e_i,e_j) = 
 \begin{pmatrix}
  2y - 2y^{-1} + y^{-1} & 0 & 0 & -y^{-1} + y^{-3}  \\
  0 & 2y^3 - 2 y +y^{-1} & y - y^{-1} & 0 \\
  0 &y - y^{-1} & y^{-1} & 0 \\
  -y^{-1} + y^{-3} & 0 & 0 & y^{-3} 
 \end{pmatrix} \,.
\end{equation*}

The Nijenhuis tensor $(N_{ij})$ in the $e_i/e^j$ basis is given by 
\begin{equation*}
 \begin{pmatrix} 
  0 & (2y^{\!-4} \! \! - \! y^{\!-6})(e_2\! + \!e_3) & -(y^{\!-4} \! \! - \! y^{\!-6})(e_2 \! + \!e_3) & y^{\!-4} (e_1 \! - \! e_4) \\
  - (2y^{\!-4} \! \! - \! y^{\!-6})(e_2\! + \!e_3)  & 0 & -y^{\!-2} (e_1 \! - \! e_4) & -(y^{\!-4} \! \! - \! y^{\!-6}) (e_2 \! + \!e_3) \\
  (y^{\!-4} \! \! - \! y^{\!-6})(e_2 \! + \!e_3) & y^{\!-2} (e_1 \! - \! e_4) & 0 & -y^{\!-6}(e_2 \! + \! e_3) \\
  -y^{\!-4} (e_1 \! - \! e_4) & (y^{\!-4} \! \! - \! y^{\!-6}) (e_2 \! + \!e_3) & y^{\!-6}(e_2 \! + \! e_3) & 0
 \end{pmatrix}
\end{equation*}
from which its norm can be computed with a bit of work. The leading order turns out to be $y^{-5}$ or equivalently $t^{-1} $ as in the Kodaira--Thurston case.

\section{Outlook}

It has been conjectured in \cite{StreetsTian} that SCF exists for as long as long as the cohomology class of $\omega(t)$ stays inside the symplectic cone  $\mathcal{C} \subset H^2(X,\bR)$. In the case of left-invariant almost K\"ahler structures on Lie groups, the tangent bundle is trivial and the first Chern class, represented by a multiple of $P$ vanishes. This means that the symplectic class is stable under SCF and the conjecture then says that the flow should exist for all times. We have confirmed the long time existence for the examples examined in the second part of this article and it would be interesting to see whether this is true in general for SCF on left-invariant almost K\"ahler structures on Lie groups. In any case, one might hope to express the limiting structure or the singularity formation in terms of the initial data, ideally of the symplectic class and the Lie algebra. \par
Is is known that, topologically, compact Nilmanifolds are iterated torus bundles (cf \cite{Nilmanifolds}). In the cases examined in this article, it seems that --- in some imprecise sense --- these fibres collapse under SCF. Studying the interaction between the iterated bundle structure and the flow might help answer the questions on the limiting structures and long time existence of SCF. \par
There are other examples of symplectic manifolds with vanishing first Chern class coming from hyperbolic geometry. They are resolutions of orbifold twistor spaces of hyperbolic four-orbifolds (cf. \cite{FinePanov2} and \cite{FinePanov3}). It would be interesting to study SCF and confirm the conjecture in these cases. \par

\printindex

\end{document}